\documentclass[a4 paper,10pt]{amsart}
\usepackage{amsfonts}
\usepackage{amssymb}
\usepackage{mathrsfs}
\usepackage[all]{xy}
\usepackage{bbm}
\usepackage{lineno}

\newtheorem{theorem}{Theorem}[section]
\newtheorem{corollary}[theorem]{Corollary}
\newtheorem{lemma}[theorem]{Lemma}
\newtheorem{proposition}[theorem]{Proposition}
\newtheorem{definition}[theorem]{Definition}
\numberwithin{equation}{section}
\newtheorem{remark}[theorem]{Remark}

\newcommand{\RR}{\mathbb{R}}
\newcommand{\NN}{\mathbb{N}}
\newcommand{\QQ}{\mathbb{Q}}
\newcommand{\ZZ}{\mathbb{Z}}

\newcommand{\cS}{\mathcal{S}}
\newcommand{\cR}{\mathcal{R}}

\newcommand{\cL}{\mathcal{L}}
\newcommand{\cH}{\mathcal{H}}

\newcommand{\mB}{\mathscr{P}}
\newcommand{\sB}{\mathscr{B}}
\newcommand{\sV}{\mathscr{C}}
\newcommand{\SL}{\mathrm{SL}}

\newcommand{\ba}{\mathbf{a}}

\newcommand{\bp}{\mathbf{p}}

\newcommand{\bx}{\mathbf{x}}

\newcommand{\bz}{\mathbf{z}}
\newcommand{\bzero}{\mathbf{0}}

\newcommand{\bv}{\mathbf{v}}

\newcommand{\bw}{\mathbf{w}}

\newcommand{\diag}{\mathrm{diag}}
\newcommand{\dist}{\mathrm{dist}}
\newcommand{\Vol}{\mathrm{Vol}}

\newcommand{\Bad}{\mathbf{Bad}}
\newcommand{\rr}{\mathbf{r}}
\newcommand{\tpp}{\mathbf{\tilde{p}}}
\newcommand{\tP}{\tilde{P}}
\newcommand{\tp}{\tilde{p}}
\newcommand{\tq}{\tilde{q}}
\newcommand{\ta}{\tilde{a}}

\def\vol{\operatorname{vol}}

\makeatletter

\begin{document}

\title[Weighted badly approximable vectors]{Weighted badly approximable vectors and games}
\date{January 10, 2017}

\author{Lifan Guan}
\address{Beijing International Center for Mathematical Research, Peking University, No. 5 Yiheyuan Road,
Beijing 100871, China.}
\email{guanlifan@pku.edu.cn}

\author{Jun Yu}
\address{Beijing International Center for Mathematical Research, Peking University, No. 5 Yiheyuan Road,
Beijing 100871, China.}
\email{junyu@math.pku.edu.cn}

\begin{abstract}
Let $d\ge 2$. We show that the set $\Bad(\rr)$ of $\rr$-badly
approximable vectors in $\mathbb{R}^{d}$ is  hyperplane absolute winning, hence is $1/2$-winning
for certain one-dimensional family of weights $\rr$.


\end{abstract}

\subjclass[2010]{11J20}

\keywords{Diophantine approximation, badly approximable vectors, $(\alpha,\beta)$-game, hyperplane absolute game.}

\maketitle

\tableofcontents

\section{Introduction}\label{S:introduction}

This paper is concerned with the study of weighted badly approximable vectors, which are natural
  generalization in high dimension of the classical badly approximable numbers. Its definition is as follows. For $d\in\mathbb{N}$,
 a tuple $\rr=(r_1,\ldots,r_d)$ belonging to the set $$\cR_d=\{\rr=(r_1,\ldots,r_d): r_i\ge0,
\sum_{i=1}^{d}r_i=1\},$$ and  $\epsilon>0$, set \begin{equation*}\label{Bad-c}\Bad_{\epsilon}(\rr)=\{(x_1,\dots,x_{d})
\in\mathbb{R}^{d}:\inf_{q\in\mathbb{N}}\max_{1\leq i\leq d}q^{r_i}\|qx_{i}\|\geq \epsilon\},\end{equation*}
where $\|\cdot\|$ means the distance of a real number to its nearest integer. Put $$\Bad(\rr)=\bigcup_{\epsilon>0}\Bad_{\epsilon}(\rr),$$
which is called the set of \emph{$\rr$-badly approximable vectors} in $\RR^{d}$. The tuple $\rr$ is called
a \emph{weight}. We also denote $\Bad(1/d,\ldots,1/d)$ simply by $\Bad_d$.

Studying the intersection of $\Bad(\rr)$ for different weights $\rr$ is a very appealing
subject which is undergoing rapid progress in recent years.
In \cite{BPV} Badziahin, Pollington and Velani proved a 30-year-old
conjecture of Schmidt \cite{Sc3} that
$$\Bad(\frac{1}{3},\frac{2}{3})\cap\Bad(\frac{2}{3},\frac{1}{3})\neq\emptyset.$$
Actually the main result of \cite{BPV} is much stronger. They proved that
\begin{theorem}[\cite{BPV}]\label{BPV-condition}
Let $\cS$ be a countable subset of $\cR_2$ satisfying
\begin{equation*}\label{E:BPV1}
\dist(\cS\setminus\partial\cR_2,\partial\cR_2)>0.
\end{equation*}
Then \begin{equation*}\label{E:BPV2}
\dim_H\bigcap_{\rr\in \cS}\Bad(\rr)=2,
\end{equation*}
where $\dim_H$ means the Hausdorff dimension.
\end{theorem}

Afterwards,  Beresnevich (\cite{Beresnevich}) proved that

\begin{theorem}[\cite{Beresnevich}]
For any $d\ge2$, let $\cS$ be a countable subset of $\cR_d$ satisfying
\begin{equation}\label{E:Ber1}
\dist(\cS\setminus \partial\cR_d,\partial\cR_d)>0.
\end{equation}
Then
\begin{equation}\label{E:Ber2}
\dim_H\bigcap_{\rr\in \cS}\Bad(\rr)=d.
\end{equation}
\end{theorem}

In the 1960s, Schmidt introduced the $(\alpha,\beta)$-game, which is played on metric spaces and whose winning sets,  the $\alpha$-winning sets ($\alpha\in(0,1)$), have the following remarkable properties:
\begin{itemize}
\item the intersection of countably many $\alpha$-winning sets is still $\alpha$-winning,
\item any $\alpha$-winning subset of a Riemannian manifold is of full Hausdorff dimension.
\end{itemize}
Schmidt also proved that
\begin{theorem}\label{schmidt}
For any $d\in \NN$, $\Bad_d$ is $1/2$-winning.
\end{theorem}
In \cite{Kl}, Kleinbock raised a question that whether $\Bad(\rr)$ is a winning set for any $d\ge 2$ and any weight $\rr\in \cR_d$. In view of the properties of the winning sets listed above, the conclusion \eqref{E:Ber2} still holds without the technical condition \eqref{E:Ber1} if Kleinbock's question has a positive answer.
In \cite{An2},  An answered Kleinbock's question positively in case $d=2$
 by proving the following theorem.
\begin{theorem}[\cite{An2}]
For any $\rr\in\cR_2$, $\Bad(\rr)$ is $(24\sqrt{2})^{-1}$-winning.
\end{theorem}
When $d>2$, nothing in this direction is known except Schmidt's classical result, namely Theorem \ref{schmidt}.

In this paper, we use a variant of the $(\alpha,\beta)$-game, namely the hyperplane absolute game introduced in \cite{BFKRW} to study $\Bad(\rr)$ for certain weights $\rr$ in high dimension. Any hyperplane winning set is $1/2$-winning (see Proposition \ref{P:HAW}). More details about the $(\alpha,\beta)$-game, the hyperplane
absolute game and their winning sets are given in Section \ref{S:game}. Recently, it was shown in \cite{BFKRW} that for any $d\in \NN$, $\Bad_d$ is a hyperplane absolute winning set. It was also shown  in \cite{NS} that for any $\rr\in\cR_2$, $\Bad(\rr)$ is a hyperplane absolute winning set.


For  $d\geq 2$, consider a subset of $\cR_d$ defined as follows
\begin{equation}\label{Rprime-d}\cR'_d:=\{\rr=(r_1,\ldots,r_d)\in\cR_d: \#\{i:r_i=\max_{1\le j\le d}r_j \}\ge d-1 \}.
\end{equation}
We prove the following theorem, which represents the first progress towards Kleinbock's question in high dimension.
\begin{theorem}\label{T:main theorem}
For any $\rr\in\cR'_d$,  $\Bad(\rr)$ is a hyperplane absolute winning set.
\end{theorem}

An immediate corollary is as follows.
\begin{corollary}
Let $\cS$ be a countable subset of $\cR'_d$. Then
$$\dim_H\bigcap_{\rr\in \cS}\Bad(\rr)=d.$$
\end{corollary}


\noindent {\it Relation with homogeneous dynamics.}
Based on works of Dani \cite{Da1} and Kleinbock \cite{Kl}, it is now well known that weighted badly approximable
vectors are closely related to bounded orbits in homogeneous dynamics. Precisely, set $G=\SL_{d+1}(\RR)$
and $\Gamma=\SL_{d+1}(\ZZ)$. For $\rr\in \cR_d$, let
\begin{equation*}
\text{$F_{\rr}^+$ be the semigroup $\{\diag(e^{r_1t},\ldots,e^{r_dt},e^{-t}):t>0\}$};
\end{equation*}
and for $\bx\in\RR^{d}$, set
\begin{equation*}
u_{\bx}=\begin{pmatrix}I_{d} & \bx \\ 0 & 1 \end{pmatrix}\in G.
\end{equation*}
Then we have the following Dani-Kleinbock correspondence \cite[Theorem 2.5]{Kl}.

\begin{proposition}\label{P:approximation-orbit}
For $\bx\in\mathbb{R}^{d}$, $\bx\in\Bad(\rr)$ if and only if the orbit
$F_{\rr}^+ u_{\bx}\Gamma$ is bounded in $G/\Gamma$.
\end{proposition}

Put \begin{equation*}
\text{$E(F_{\rr}^+)=\{x\in G/\Gamma : F_{\rr}^+x$ is bounded in $G/\Gamma\}$.}
\end{equation*}

As noted at the end of \cite{AGK}, using the
methods developed in \cite{AGK} one can gain information about the subset $E(F_{\rr}^+)\subset G/\Gamma$ from
properties of $\Bad(\rr)\subset\RR^{d}$. Hence, Theorem \ref{T:main theorem} may help to verify some special
cases of \cite[Conjecture 7.1]{AGK}.

\smallskip

\noindent {\it Organization of the paper.} In Section 2, we first recall the definitions and basic properties of various games. Then in Section 2.4, we reduce the proof of Theorem \ref{T:main theorem} to a concrete lemma (Lemma \ref{l-b-hpw}). In  Section 3, we first attach a rational hyperplane $\cH_P$ to each rational point, then we define a decomposition of $\QQ^n$ using the attached hyperplane. This decomposition is a direct generalization of the decomposition introduced by An in \cite{An2}, which plays a central role there. At the end of  Section 3, we come up with  the most important new ingredient in this paper, that is, we attach a line $\cL_{P}$ with well-chosen
bounds on its defining coefficients to each rational point $P$ in a certain class (Lemma \ref{L:v-existence}). Only with such careful chosen bounds and lines, the
proof goes through. The last two sections are devoted to proving Lemma \ref{l-b-hpw}.


\smallskip

\noindent {\it Acknowledgements.} We thank Jinpeng An for helpful discussions. Jun Yu is partially supported by the Recruitment Program of Global Young Experts of China.


\section{Games}\label{S:game}

In this section we first recall some basics of the $(\alpha,\beta)$-game, the hyperplane absolute game,
the hyperplane potential game and their winning sets. See \cite{An1}, \cite{AGK}, \cite{BFKRW}, \cite{FSU},
\cite{Mc}, \cite{Sc1} for more details. Then  we reduce the proof of Theorem \ref{T:main theorem} to a concrete lemma. We confine our discussion to subsets of a Euclidean
space $\mathbb{R}^{d}$. Let $\rho(B)$ denote the radius of a closed ball $B$.

\subsection{$(\alpha,\beta)$-game}\label{SS:Schmidt}

In \cite{Sc1}, Schmidt introduced the $(\alpha,\beta)$-game. Being played on $\RR^d$, the game involves two
parameters $\alpha,\beta\in (0,1)$, a target set $S\subset \RR^{d}$ and two players Alice and Bob. Let $i\ge 0$, at the $i$-th round, Bob chooses a closed ball $B_i$ with $\rho(B_{i})=\beta\rho(A_{i-1})$
(an arbitrary ball in case $i=0$), and Alice chooses a closed ball $A_i\subseteq B_i$ with
$\rho(A_i)=\alpha\rho(B_i)$. By this process there is a nested sequence of closed balls
\begin{equation*}
B_0\supseteq A_0\supseteq B_1\supseteq A_1\supseteq  B_2\supseteq \ldots \quad .
\end{equation*}
$S$ is called \emph{$(\alpha,\beta)$-winning} if Alice has a winning strategy ensuring that
$$\bigcap_{i=0}^\infty B_i\cap S\neq \emptyset,$$ regardless of how Bob chooses to play.
For an $\alpha\in(0,1)$,  $S$  is called \emph{$\alpha$-winning} if it is
$(\alpha,\beta)$-winning for any $\beta\in(0,1)$.

No proper subset of $\RR^{d}$ is $(\alpha,\beta)$-winning if $1-2\alpha+\alpha\beta\leq 0$
(\cite[Lemma 5]{Sc1}). The $\alpha$-winning sets enjoy many properties (\cite{Sc2}):
\begin{enumerate}
\item Given $\alpha,\alpha'\in(0,1)$, if $\alpha\geq\alpha'$, then an $\alpha$-winning set
is also $\alpha'$-winning. If $\alpha>1/2$, then no proper subset of $\mathbb{R}^{d}$
is $\alpha$-winning.
\item The intersection of countably many $\alpha$-winning sets is again an $\alpha$-winning set.
\item If $S$ is an $\alpha$-winning set, then $S$ is thick. Recall that a subset $S$ of $\mathbb{R}^{d}$
is \emph{thick} if its intersection with any nonempty open subset of $\mathbb{R}^{d}$ has full
Hausdorff dimension.
\item Let $\varphi:\RR^{d}\to \RR^{d}$ be a bi-Lipschitz homeomorphism. If $S$ is an $\alpha$-winning
set, then $\varphi(S)$ is $\alpha'$-winning for some $\alpha'$ depending on $\varphi$ and $\alpha$.
\end{enumerate}
As such, the $(\alpha,\beta)$-game has been a powerful tool for proving full dimensionality
(and non-emptiness) of fractal sets (\cite{An1}, \cite{An2}, \cite{An-B-V}, \cite{AGK}).

\subsection{Hyperplane absolute game}\label{SS:HAW}

The hyperplane absolute game was introduced in \cite{BFKRW}. It is played on a Euclidean space $\RR^{d}$.
Given a hyperplane $\cH$ and $\delta>0$, denote by $\cH^{(\delta)}$ the $\delta$-neighborhood of $\cH$,
$$\cH^{(\delta)}=\{\bx\in\RR^{d}:\mathrm{dist}(\bx,\cH)\leq\delta\}.$$ For $\beta\in(0,1/3)$, The
\emph{$\beta$-hyperplane absolute game} involves a parameter $\beta\in(0,1/3)$, a target set $S\subset \RR^{d}$ and two players Alice and Bob.
Let $i\ge 0$, at the $i$-th round, Bob chooses a closed ball $B_i$ of radius $\rho_i$
such that $B_{i}\subseteq B_{i-1}\setminus \cH_{i-1}^{(\delta_{i-1})}$ and $\rho_{i}\geq\beta\rho_{i-1}$
(an arbitrary ball in case $i=0$), and
Alice chooses a hyperplane neighborhood $\cH_i^{(\delta_i)}$ with $\delta_i\le\beta \rho_{i}$.
By this process there is a nested sequence of closed balls $$B_0\supseteq B_1
\supseteq B_2\supseteq \ldots\quad .$$ $S$ is called \emph{$\beta$-hyperplane absolute
winning} (\emph{$\beta$-HAW} for short) if Alice has a winning strategy ensuring that
$$\bigcap_{i=0}^\infty B_i\cap S\neq \emptyset,$$ regardless of how Bob chooses to play.  $S$ is called
\emph{hyperplane absolute winning}
(\emph{HAW} for short) if it is $\beta$-HAW for any $\beta\in(0,1/3)$.

We have the following properties of $\beta$-HAW sets and HAW sets (\cite{BFKRW}, \cite{KW3}),
\begin{enumerate}
\item Given $\beta,\beta'\in(0,1/3)$, if $\beta\geq\beta'$, then any $\beta'$-HAW set is also
$\beta$-HAW.
\item An HAW subset is $\alpha$-winning for any $\alpha$, $0<\alpha<1/2$.
\item The intersection of countably many HAW sets is again HAW.
\item Let $\varphi:\RR^{d}\to \RR^{d}$ be a $C^1$ diffeomorphism. If $S$ is an HAW set, then so is
$\varphi(S)$.
\end{enumerate}

We prove a strengthening of (2), which might be known to experts. It is of independent interest.

\begin{proposition}\label{P:HAW}
Given $\alpha,\beta\in (0,1)$, if $1-2\alpha+\alpha\beta>0$, then any HAW set is $(\alpha,\beta)$-winning.
In particular, any HAW set is $1/2$-winning.
\end{proposition}

\begin{proof}
 For any $\alpha,\beta\in (0,1)$ satisfying $1-2\alpha+\alpha\beta>0$,
write $\theta=1-2\alpha+\alpha\beta$. Choose an $N\in\NN$ satisfying \begin{equation}\label{ine-theta}
(\alpha\beta)^N < \frac{1}{3}\theta.
\end{equation}
Write \begin{equation}\label{def-betab}
\beta'=\frac{1}{2}(\alpha\beta)^N.
\end{equation}
We are going to prove that $\beta'$-HAW implies $(\alpha,\beta)$-winning, from which the conclusion follows.
Let $S$ be an HAW subset of $\RR^d$.
Let $B_i=B(y_i,\rho_i)$ (resp. $A_i=B(x_i,\alpha\rho_i)$) be Bob's (resp. Alice's) choice at the $i$-th round
of the $(\alpha,\beta)$-game. Then the sequences of balls $\{B(y_i,\rho_i)\}, \{B(x_i,\alpha\rho_i)\}$ should
satisfy the following conditions,
\begin{equation}\label{center-radius}
\dist(x_i,y_i)\le(1-\alpha)\rho_i,\quad \dist(x_i,y_{i+1})\le \alpha(1-\beta)\rho_i, \quad
\rho_{i+1}=\alpha\beta \rho_i.
\end{equation}

We are going to construct a corresponding $\beta'$-hyperplane absolute game, in which Bob's choice at the $k$-th
round is the ball $B_{kN}$ chosen by himself at the $(kN)$-th round in the $(\alpha,\beta)$-game and Alice's choice
at the $k$-th round is a hyperplane neighborhood $\cH_{k}^{(\delta_k)}$ chosen according to her winning strategy.
Once such a game is constructed, we obviously have the outcome point $$x_{\infty}\in \bigcap_{k=0}^\infty B_{kN}$$
belonging to $S$. Thus $S$ is $(\alpha,\beta)$-winning. By definition, to construct such a game, we only need to
make sure that:
\begin{itemize}
\item[(a)] $\rho_{(k+1)N}\ge \beta' \rho_{kN}$.
\item[(b)] $B_{(k+1)N}\subset B_{kN}\setminus \cH_k^{(\delta_k)}$.
\end{itemize}
Since $\rho_{(k+1)N}=(\alpha\beta)^N\rho_{kN}$, (a) follows directly from \eqref{def-betab}.
Now we claim that if Alice chooses her ball $A_i$ ($kN\le i <(k+1)N)$ as far away from $\cH_k$ as possible,
then we have (b) as a consequence.

Indeed, for each $kN\le i <(k+1)N$, Alice can choose $x_i$ with
\begin{equation}\label{ine-x-hp}
 \dist(x_i,\cH_k)-\dist(y_i,\cH_k)=(1-\alpha)\rho_{i}.
\end{equation}
According to \eqref{center-radius}, no matter how Bob makes his choice, we always have
\begin{equation}\label{ine-y-hp}
  \dist(y_{i+1},\cH_k)-\dist(x_{i},\cH_k)\ge -\alpha(1-\beta)\rho_{i}.
\end{equation}
It follows from the inequalities \eqref{ine-x-hp} and \eqref{ine-y-hp} that
\begin{equation}\label{ine-hp-dist}
\dist(y_{i+1},\cH_k)-\dist(y_{i},\cH_k)\ge (1-\alpha-\alpha(1-\beta))\rho_{i}=\theta\rho_i.
\end{equation}
Summing up the above inequalities \eqref{ine-hp-dist} for all $i$ ($kN\le i <(k+1)N$), we get
\begin{equation}\label{ine-y-final}
\dist(y_{(k+1)N},\cH_k)\ge \theta\sum_{i=kN}^{(k+1)N-1}\rho_i\ge\theta\rho_{kN}.
\end{equation}
According to \eqref{ine-theta} and \eqref{def-betab}, we have
\begin{equation}\label{ine-fi}
\rho_{(k+1)N}+\delta_k\le ((\alpha\beta)^N+\beta')\rho_{kN} <(\frac{1}{3}\theta+\frac{1}{6}\theta)\rho_{kN}
=\frac{1}{2}\theta\rho_{kN}.
\end{equation}
Then, (b) follows from \eqref{ine-y-final} and \eqref{ine-fi} immediately.
\end{proof}

\subsection{Hyperplane potential game}\label{SS:HPW}

Being introduced in \cite{FSU}, the hyperplane potential game also defines a class of subsets of $\RR^{d}$
called \emph{hyperplane potential winning} sets. By the following proposition these
two classes of sets are indeed the same (\cite[Theorem C.8]{FSU}).

\begin{proposition}\label{P:HPW}
A subset $S$ of $\RR^{d}$ is hyperplane potential winning if and only if it is hyperplane absolute winning.
\end{proposition}
\noindent As such, the hyperplane potential game is a powerful tool for proving the HAW property
since the hyperplane potential game is more flexible than the hyperplane absolute game in some
aspects.

The hyperplane potential game involves two parameters $\beta\in(0,1)$, $\gamma>0$, a target set $S\subset \RR^d$ and two players Alice and Bob.
Let $i\ge 0$, at the $i$-th round,
Bob chooses a closed ball $B_i$ of radius $\rho_i$ such that $\rho_{i}\geq\beta\rho_{i-1}$ (an arbitrary
ball in case $i=0$), and Alice chooses a countable family of hyperplane neighborhoods
$\{\cH_{i,k}^{(\delta_{i,k})}: k\in \NN\}$ such that
\begin{equation}\label{HPW ine}
\sum_{k=1}^\infty \delta_{i,k}^\gamma\le(\beta \rho_{i})^\gamma.
\end{equation}
By this process there is a nested sequence of closed balls $$B_0\supseteq B_1\supseteq B_2\supseteq\ldots\quad .$$
$S$ is called \emph{$(\beta,\gamma)$-hyperplane potential winning}
(\emph{$(\beta,\gamma)$-HPW} for short) if Alice has a winning strategy ensuring that
$$\bigcap_{i=0}^\infty B_i\cap\Big(S\cup\bigcup_{i=0}^\infty\bigcup_{k=1}^\infty \cH_{i,k}^{(\delta_{i,k})}\Big)
\ne\emptyset,$$ regardless of how Bob chooses to play. $S$ is called \emph{hyperplane potential winning}
(\emph{HPW} for short) if it is $(\beta,\gamma)$-HPW for any $\beta\in(0,1)$ and $\gamma>0$.

\subsection{Reduction of Theorem \ref{T:main theorem}.} For the proof of Theorem \ref{T:main theorem}, we may fix $$d\ge 2, \text{ and } \rr\in \cR_d'$$ from now on.
Then it is convenient to introduce the following notation.
\begin{definition}
Let $B\subset \RR^d$ be a closed ball, $\beta\in(0,1)$ and $\gamma>0$. Say a subset $S\subset \RR^d$ is \emph{$(B,\beta,\gamma)$-HPW} if Alice can win the $(\beta,\gamma)$-hyperplane potential game whenever Bob chooses $B$ as his $B_0$.
\end{definition}
Theorem \ref{T:main theorem} can be deduced from the following lemma.
\begin{lemma}\label{l-b-hpw}
For any closed ball $B_0\subset \RR^d$ of radius $\rho_0\le 1$, any $\beta\in(0,1)$ and any $\gamma>0$, the set $\Bad(\rr)$ is $(B_0,\beta,\gamma)$-HPW.
\end{lemma}

\begin{proof}[Proof of Theorem \ref{T:main theorem} modulo Lemma \ref{l-b-hpw}.]
In view of Proposition \ref{P:HPW}, to prove Theorem \ref{T:main theorem}, it suffices to prove that for any $\beta\in(0,1)$ and any $\gamma>0$, Alice has a winning strategy for the $(\beta,\gamma)$-hyperplane potential game played on $\RR^d$ with target set $\Bad(\rr)$. Denote the closed ball chosen by Bob at the $i$-th round as $B_i$ with radius $\rho_i$. By \cite[Remark 2.4]{AGK}, we may assume $\rho_0\le 1$ without loss of generality, which completes the proof of Theorem \ref{T:main theorem} modulo Lemma \ref{l-b-hpw}.
\end{proof}

\begin{remark}
The rest of the paper is devoted to proving Lemma \ref{l-b-hpw}. From now on, we fix a triple  $\Phi$ which consists of
\begin{itemize}
\item a closed ball $B_0\subset\RR^d$ of radius $\rho_0\le 1$,
\item a number $\beta\in (0,1)$,
\item a number $\gamma>0$.
\end{itemize}
\end{remark}
\section{Attaching a hyperplane and a line}\label{S:line}

Given a lattice $L$ in $\RR^d$, let $d(L)=\vol(\RR^{d}/L)$ denote the covolume of $L$. We shall need the following version of  Minkowski's linear forms theorem (\cite[Theorem 2C]{Sc2}).

\begin{theorem}\label{T:Minkowski}
Let $d\ge 2$. Given linearly independent linear forms $l_1,\ldots,l_d$ on $\RR^d$, let $(l_1,\ldots,\l_d)$
denote the linear transform from $\RR^{d}$ to itself generated by the linear forms $l_1,\ldots,l_d$.
For any lattice $L$ in $\RR^d$ and positive numbers $A_1,\ldots,A_d$ satisfying
\begin{equation*}
A_1\cdots A_d\geq |d(L)| \cdot|\det(l_1,\ldots,\l_d)|,
\end{equation*}
there exists $\bx\in L\setminus \{\bzero\}$ such that
\begin{equation*}
|l_1(\bx)|\leq A_1 \quad \text{ and } \quad |l_i(\bx)|< A_i \quad (2\le i\le d).
\end{equation*}
\end{theorem}

We make the following convention throughout this paper. Whenever we write a rational point $P\in \QQ^d$ as
\begin{equation*}\label{P}
P=\frac{\bp}{q} \text{ with } \bp=(p_1,\ldots,p_d),
\end{equation*}
we always mean \begin{equation*}
q>0 \quad \text{and} \quad (p_1,\ldots,p_d,q)=1.
\end{equation*}
Such a form is unique. Then, the denominator $q$ is a function of $P$. We write it as $q=q(P)$.

\subsection{Attaching a hyperplane}\label{SS:def-fp} For convenience, we set
\begin{equation}\label{s}
s=s(\rr)=\max_{1\le i\le d}r_i, \text{ and choose } i_0\in \{j: r_j=\min_{1\le i\le d}r_i\}.
\end{equation}
For each rational point $P=\bp/q\in \QQ^d$, we define a lattice
\begin{equation}\label{Lambda-P}
\Lambda_P=\{b\frac{\bp}{q}+\bz: \quad b\in \ZZ,\bz\in \ZZ^d\}.
\end{equation}
Since $[\Lambda_P:\ZZ^d]=q$, we have $d(\Lambda_P)=1/q $. Hence $d(\Lambda_P^*)=q $, where $\Lambda_P^*$
means the dual lattice of $\Lambda_P$. Note that
\begin{equation*}
\Lambda_P^*=\{\ba\in\ZZ^d: \ba\cdot\bp\in q\ZZ\}.
\end{equation*}
Choose and fix
\begin{equation}\label{def-good}
\ba_P\in X_P=[-q^{r_1},q^{r_1}]\times \cdots \times [-q^{r_d},q^{r_d}] \cap \Lambda_P^*\setminus \{\bzero\}.
\end{equation}
The non-emptiness of $X_P$ is ensured by Theorem \ref{T:Minkowski}.

Define an affine form
\begin{equation}\label{affine-form}
F_P(\bx)=\ba_P\cdot\bx+C_P, \text{ where }C_P=-q^{-1}\ba_P\cdot\bp\in \ZZ,
\end{equation}
and a hyperplane
\begin{equation}\label{Hyperplane}
\cH_P=\{\bx\in\RR^d: F_P(\bx)=0\}.
\end{equation}
It is obvious that $P\in \cH_P$.

Write $\ba_{P}=(a_1,\dots,a_{d})$. We define two $\NN$-valued functions on $\QQ^d$,
\begin{equation}\label{xi-P} \xi_P=\xi(P)=\max\{|a_{i}|: 1\leq i\leq d\}\end{equation}
and \begin{equation}\label{H-P} H(P)=q(P)\xi(P).\end{equation} According to \eqref{def-good},
\begin{equation}\label{ine-xi-h}\xi_P\le q(P)^s \quad \text{ and } \quad  H(P)\le q(P)^{1+s}.\end{equation}
Note that \begin{equation}\label{bound-a} |a_{i}|\leq\min\{q(P)^{r_{i}},\xi_{P}\} \text{ for any   $1\leq i\leq d$}\end{equation}

\subsection{Constants and subdivisions}\label{SS:subdivision} In this subsection, we introduce some constants and subdivisions.
For $n\geq 1$, let $\sB_n$ be the set of closed balls defined by
\begin{equation}\label{def-b-n}
\sB_n=\{B\subset B_0:
\beta R^{-n}\rho_0<\rho(B)\le R^{-n}\rho_0\},
\end{equation} where $R$ ia a positive number satisfying
\begin{equation}\label{Rvalues}(R^\gamma-1)^{-1}\leq (\frac{\beta^2}{2})^\gamma.\end{equation}
Note that this implies \begin{equation}\label{R-large}R>2\beta^{-2}>\max\left\{2,\beta^{-1}\right\},
\end{equation} which shows that those $\sB_n$ are mutually disjoint.

Now we define a decomposition of $\QQ^d$. Write
\begin{equation}\label{values}
c=\frac{1}{8}d^{-2}\rho_0 R^{-18d^2} \quad \text{ and } \quad H_n=dc\rho_0^{-1}R^{n}.
\end{equation}
Put \begin{equation}\label{B-n}\mB_n=\{P\in\QQ^{d}: H_n\le H(P) <H_{n+1} \}.\end{equation}
Set \begin{equation}\label{B-n1}\mB_{n,1}=\{P\in\mB_{n}: H_n^{\frac{1}{1+s}}\leq q(P)<
H_n^{\frac{1}{1+s}}R^{12d^2}\}.\end{equation} For $k\geq 2$, write \begin{equation}\label{Q-nk}
Q_{n,k}=H_n^{\frac{1}{1+s}}R^{d(k-2)+12d^2}\end{equation} and set \begin{equation}\label{B-nk}
\mB_{n,k}=\{P\in\mB_{n}: Q_{n,k} \leq q(P) < Q_{n,k+1} \}.\end{equation}

 By definition, those $\mB_{n,k}$ are mutually disjoint. The following lemma summarizes some basic properties of this decomposition.
\begin{lemma}\label{L:decomposition}
\
\begin{enumerate}
\item $\QQ^{d}=\bigsqcup_{n=1}^{\infty}\bigsqcup_{k=1}^{n-1}\mB_{n,k}$.
\item For $P\in \mB_{n,k}$ with $k\ge 2$, we have
\begin{equation}\label{ine qxi}
\psi_P:=q^{-1-s}H(P)\leq R^{-dk-10d^2}.
\end{equation}
\end{enumerate}
\end{lemma}

\begin{proof}
As $H_1<1$, it follows that $\QQ^{d}=\bigsqcup_{n=1}^{\infty}\mB_n$. By definition, $\mB_n=\bigsqcup_{k=1}^{\infty} \mB_{n,k}$.
Hence to prove (1), it suffices to prove that $\mB_{n,k}=\emptyset$ for $k\ge n$. We argue by contradiction. Note that
$H_{1}<1$, so we may assume that $\mB_{n,k}\ne \emptyset$ for some $n$ and $k$ with $k\geq n\geq 1$.
Taking $P\in \mB_{n,k}$, then by definition we have
\begin{equation*}
H_n^{\frac{1}{1+s}}R^{d(k-2)+12d^2}=Q_{n,k}\leq q(P)< H_{n+1}=H_{n}R.
\end{equation*}
Since $H_{n}\leq R^{n}$ and $s\leq1/(d-1)$, we have
\begin{equation*}
R^{n}\geq H_{n}\geq R^{\frac{s+1}{s}(d(k-2)+12d^2-1)}\geq R^{d(d(k-2)+12d^2-1)} \geq R^{d^{2}k}.
\end{equation*}
Then $k<n$, which leads to a contradiction.

The inequality \eqref{ine qxi} is verified by a direct computation,
\begin{equation*}
\psi_P=q^{-1-s}H(P)\leq Q_{n,k}^{-1-s}H_{n+1}=R^{1-(1+s)(d(k-2)+12d^2)}\le R^{-dk-10d^2}.
\end{equation*}
\end{proof}

\begin{remark}
We use $k_P$ to denote the unique number such that
$P\in \mB_{*,k_P}$, which is well-defined by Lemma \ref{L:decomposition}(1).
\end{remark}

\subsection{Attaching a line}\label{SS:attach-line}

In this subsection, we attach a suitable rational line to each rational point $P$ with $k_P\ge 2$.
We begin with a lemma.
\begin{lemma}\label{L:v-existence}
Let $P=\bp/q \in \QQ^d$ and $\ba_P=(a_1,\ldots,a_d)$. Set
\begin{equation}\label{def-w-p}
\bw_P=(w_1,\ldots,w_d) \text{ with }w_i=1 \text{ for } i\ne i_0, \text{ and }w_{i_0}=\psi_P,
\end{equation}
where $i_0$ is given in \eqref{s}. Then there exists $\bv=(v_1,\ldots,v_d)\in \Lambda_P \setminus \{\bzero\}$ such that
\begin{equation}\label{ine-def-v}
|v_i|\le (d-1)w_iq^{-r_i} \quad \text{for each } 1\le i\le d.
\end{equation}
\end{lemma}

\begin{proof}
Let $j\in[1,d]$ be such that
\begin{equation*}
w_{j}|a_{j}|q^{-r_{j}}=\max\{w_{i}|a_{i}|q^{-r_{i}}: 1\leq i\leq d\}.
\end{equation*}
Set
\begin{equation*}
\Pi_j:=\{\bx=(x_1,\ldots,x_d)\in \RR^d:|x_i|\le w_iq^{-r_i} \text{ for } i\ne j,\  |\ba_P\cdot \bv|<1\}.
\end{equation*}
We claim that
\begin{equation}\label{vol-pi-j}
\Vol(\Pi_j)\ge q^{-1}.
\end{equation}
There are two cases.
\begin{itemize}
\item If $j\ne i_0$, then $\Vol(\Pi_j)=|a_j|^{-1}\prod_{i\ne j}w_iq^{-r_j}\ge q^{-1}$.
\item If $j=i_0$, then $\Vol(\Pi_j)=|a_j|^{-1}\prod_{i\ne j}w_iq^{-r_j}\ge \xi_P^{-1}\psi_Pq^{-1+s}\ge q^{-1}$.
\end{itemize}
  In view of Theorem \ref{T:Minkowski},  \eqref{vol-pi-j} and the fact $d(\Lambda_P)=1/q$,  there exists
  $$\bv=(v_1,\ldots,v_d)\in \Pi_j\cap \Lambda_P \setminus \{\bzero\}.$$
 Since $\ba_P\in \Lambda_P^{*}$, $|\ba_P\cdot \bv|<1$ implies $|\ba_P\cdot \bv|=0$. Hence we have the following estimate,
\begin{equation*}
|v_{j}|\leq |a_{j}|^{-1}\sum_{i\neq j}|a_iv_i|\leq |a_{j}|^{-1}\sum_{i\neq j} w_{i}\frac{|a_{i}|}{q^{r_{i}}}
\leq (d-1)w_{j}q^{-r_{j}},
\end{equation*}
which completes the proof.
\end{proof}

According to Lemma \ref{L:v-existence}, for each rational point $P$ with
$k_P\ge 2$, the set of vectors $\bv$ satisfying \eqref{ine-def-v}
is nonempty. We choose and fix one as $\bv_P$. Then we define a rational line passing through $P$ by
\begin{equation}\label{Line}
\cL_{P}=\{\bx\in \RR^d: \bx-\frac{\bp}{q}=\lambda \bv_P, \lambda\in \RR\}.
\end{equation}

\section{A key proposition}\label{S:prop}
It is easily checked that
 \begin{equation*}\label{Bad-c-equation}\Bad_{\epsilon}(\rr)=\RR^{d}\setminus\bigcup_{P\in\QQ^{d}}
\Delta_{\epsilon}(P), \end{equation*}
where for $P=(p_1/q,\ldots,p_d/q)\in \QQ^d$,
\begin{equation}\label{delta-c}
\Delta_{\epsilon}(P)=
\{\bx=(x_1,\dots,x_{d})\in\RR^{d}:|x_i-\frac{p_i}{q}|<\frac{\epsilon}{q^{1+r_i}}, i=1,2,\dots,d\}.
\end{equation}
Define a partial order $"<"$ on the set $\QQ^{d}$ by
\begin{equation*}
P<P' \Longleftrightarrow \Delta_{c}(P)\subsetneq \Delta_{c}(P'),
\end{equation*}
where the constant $c$ is given in \eqref{values}.

\begin{definition}\label{D:maximal}
A rational point $P$ is called \emph{maximal} if $P$ is a maximal element for the partial order $<$. Let
$\sV$ denote the set of maximal points.
\end{definition}

\begin{remark}
Note that the partial order defined as above depends on the constant $c$, which is determined by the triple $\Phi$. As $\Phi$ has already been fixed, we omit the dependence in the definition.
\end{remark}

The following lemma is easy but important.

\begin{lemma}\label{L:nouse}
$$\bigcup_{P\in\QQ^{d}}\Delta_{c}(P)=\bigcup_{P\in\sV}\Delta_{c}(P).$$
\end{lemma}

\begin{proof}
We need to show that: for any $P$, there is a maximal element $Q$ such that
$P<Q$. Indeed it follows directly from the definition that the set
\begin{equation}\label{p-pp}
\{P'\in \QQ^d:P<P'\}
\end{equation}
is finite. Assume the contrary that there is no maximal element $Q$ such that
$P<Q$, then it follows that the set \eqref{p-pp} is infinite, which leads to a contradiction.

\end{proof}

For a closed ball $B\in\sB_{n}$, set
 $$\sV_{n+k,k}(B)=\{P\in\mB_{n+k,k}\cap\sV:  \Delta_c(P)\cap B \ne \emptyset\}.$$

The following is a key proposition needed in the proof of Lemma \ref{l-b-hpw}.
\begin{proposition}\label{P:main prop}
Let $n\ge 1$, $B\in\sB_n$ and $k\ge 1$. Then
$$\sV_{n+k,k}(B)\subset\bigcap_{P\in\sV_{n+k,k}(B)}\cH_P.$$
\end{proposition}
To prove Proposition \ref{P:main prop}, it suffices to show that $P_1\in \cH_{P_2}$ for any $P_1, P_2 \in \sV_{n+k,k}(B)$. We pick
$$P_1=\frac{\bp_1}{q_1},P_2=\frac{\bp_2}{q_2}\in \sV_{n+k,k}(B), \text{ with } \bp_j=(p_{1,j},\ldots, p_{d,j})\text{ where } j=1,2, $$
and write $$\ba_{P_j}=(a_{1,j},\ldots,a_{d,j}) \text{ where } j=1,2. $$

Before proving Proposition \ref{P:main prop}, we give the following useful estimate.
\begin{lemma}\label{L:estimate}

For the function $F_{P}$ defined in Subsection \ref{SS:def-fp}, we have
\begin{eqnarray}\label{ine-dist}
\left|F_{P_2}(P_1)\right|\le
\begin{cases}
3d^2R^{12d^2+2}q_1^{-1}c, &\text{if } k=1 \\
3d^2R^{k+d+1}q_1^{-1}c,&\text{if } k\geq 2
\end{cases}
\end{eqnarray}
\end{lemma}

\begin{proof}
Since $P_1,P_2\in \mB_{n+k,k}(B)$, we can pick points
$$\bx=(x_1,\dots,x_{d})\in\Delta_{c}(P_1)\cap B\quad\text{ and }\quad\bx'=(x'_1,\dots,x'_{d})\in\Delta_{c}(P_2)\cap B.$$
Then,
\begin{eqnarray*}
&&\quad|F_{P_2}(P_1)|\\
&&=\Big|\sum_{1\leq i\leq d}a_{i,2}\frac{p_{i,1}}{q_1}-\sum_{1\leq i\leq d}a_{i,2}\frac{p_{i,2}}{q_2}\Big|\\
&&=\Big|\sum_{1\leq i\leq d}a_{i,2}\Big(\frac{p_{i,1}}{q_1}-x_i+x_i-x'_{i}+x'_{i}-\frac{p_{i,2}}{q_2}\Big)\Big|\\
&&\leq\sum_{1\leq i\leq d}|a_{i,2}|\Big(\frac{c}{q_1^{1+r_i}}+\frac{c}{q_2^{1+r_i}}+2R^{-n}\rho_0\Big)\\
&&\leq^{}\sum_{1\leq i\leq d} q_2^{r_i}\Big(\frac{c}{q_1^{1+r_i}}+\frac{c}{q_2^{1+r_i}}\Big)+2d\xi_{P_2}R^{-n}\rho_0 \qquad (\text{by } \ref{bound-a})\\
&&=\sum_{1\leq i\leq d}\frac{cq_2^{r_i}}{q_1^{1+r_i}}+\frac{dc}{q_2}+\frac{2dR^{-n}\rho_{0}H(P_2)}{q_2}\\
&&\leq\frac{2dc}{q_1}\max\{\frac{q_1}{q_2},1,\frac{q_2}{q_1}\}+\frac{2d^2R^{k+1}c}{q_1}\cdot\frac{q_1}{q_2}\\
&&\leq^{} \begin{cases}
(2d^2R^{k+1}+2d)R^{12d^2}q_1^{-1}c & \text{if } k=1,\\
(2d^2R^{k+1}+2d)R^{d}q_1^{-1}c &\text{if } k\ge 2
\end{cases} \qquad (\text{by definitions in Section \ref{SS:subdivision}})\\
&&\leq \begin{cases}
3d^2R^{12d^2+2}q_1^{-1}c &\text{if } k=1,\\
3d^2R^{k+d+1}q_1^{-1}c &\text{if } k\geq 2.
\end{cases}
\end{eqnarray*}
\end{proof}


\begin{proof}[Proof of Proposition \ref{P:main prop}]
 The proof is divided according to two cases.

In the case of $k=1$, by \eqref{values} and \eqref{ine-dist} we have
\begin{equation*}
q_1\Big|F_{P_2}(P_1)\Big|\le  3d^{2} R^{12d^2+2}c<1.
\end{equation*}
Since $q_1\big|F_{P_2}(P_1)\big|\in\ZZ$, then $F_{P_2}(P_1)=0$. Hence, $P_1$ lies on
$\cH_{P_2}$.

Now we assume that $k\ge 2$. Note that we have attached a rational line $\cL_{P}$
passing through $P$ to each rational vector $P$ with $k_{P}\ge 2$. Write the corresponding vector $\bv_{P_j} (j=1,2)$  as   $$\quad \bv_{P_j}=(v_{1,j},\ldots,v_{d,j}).$$
As $\bv_{P_j}\in \Lambda_{P_j}\setminus\{\bzero\}$, there are $b_j\in \ZZ, \bz_j=(z_{1,j},\ldots, z_{d,j})\in \ZZ^d$ such that
\begin{equation}\label{v-b-z}
\bv_{P_j}=b_j\frac{\bp_1}{q_1}+\bz_j, \text{ or equivalently } v_{i,j}=b_j\frac{p_{i,j}}{q_{j}}+z_{i,j} \text{ where } i=1,\ldots,d.
\end{equation}

We argue by contradiction. Suppose that $P_1$ does not lie on $\cH_{P_2}$, or equivalently,
\begin{equation}\label{assume-f-p}
F_{P_2}(P_1)\ne 0.
\end{equation} Then,
either $\cL_{P_1}$ is parallel to $\cH_{P_2}$, or $\cL_{P_1}$ intersects with $\cH_{P_2}$ at a point.

If $k\ge 2$ and $\cL_{P_1}$ is parallel to $\cH_{P_2}$, then
\begin{equation}\label{parallel-av}
\ba_{P_2}\cdot\bv_{P_1}=0.
\end{equation}
 We claim that
\begin{equation}\label{claim1}
q_1v_{i,1}F_{P_2}(P_1)\in \ZZ \quad \text{for each } 1\le i \le d.
\end{equation}
In view of \eqref{v-b-z} and the definition of $F_{P_2}$, to prove \eqref{claim1} it suffices to show
\begin{equation}\label{pr-claim1}
b_1 q_1^{-1}\ba_{P_2}\cdot\bp_1\in \ZZ.
\end{equation}
This follows easily from \eqref{parallel-av}.

As $\bv_{P_1}\ne \bzero$, it follows from \eqref{assume-f-p} and \eqref{claim1} that
\begin{equation}
q_1|F_{P_2}(P_1)|\Big(\sum_{1\le i\le d}|v_{i,1}|\Big)\ge 1.
\end{equation}
Note that, by definition we have
\begin{equation}\label{max-q-psi}
\max_{1\le i\le d}\{w_{i,1}q_1^{-r_i}\}=\max\{q_1^{-s}, \psi_{P_1}\} \le \psi_{P_1}.
\end{equation}
Hence we have
\begin{eqnarray*}
&&\quad q_1|F_{P_2}(P_1)|\Big(\sum_{1\le i\le d}|v_{i,1}|\Big)\\
&&\le^{} 3d^2R^{k+d+1}c \Big(\sum_{1\le i\le d}(d-1)w_{i,1}q_1^{-r_i}\Big) \qquad (\text{by }\eqref{ine-def-v} \text{ and }\eqref{ine-dist})\\
&&\le 3d^4R^{k+d+1}c\max_{1\le i\le d}\{w_{i,1}q_1^{-r_i}\} \\
&&\le^{} 3d^4R^{k+d+1}\psi_{P_1}c \qquad (\text{by }\eqref{max-q-psi})\\
&&\le^{} 3d^4R^{k+d+1-dk-10d^2}c \qquad (\text{by }\eqref{ine qxi})\\&&<1,
\end{eqnarray*}
which leads to a contraction.

If $k\ge 2$ and $\cL_{P_1}$ intersects with $\cH_{P_2}$, let $$P_0=\frac{\bp_0}{q_0}=(\frac{p_{1,0}}{q_0},\dots,\frac{p_{d,0}}{q_0})$$
be the point of their intersection. Write
\begin{equation}\label{p0-p1}
\frac{\bp_0}{q_0}=\frac{\bp_1}{q_1}+\lambda_0\bv_{P_1}.
\end{equation}
Now we are going to prove that
\begin{equation}\label{containment}
\Delta_c(P_1)\subsetneq \Delta_c(P_0),
\end{equation} by which we get a contradiction
from the assumption that $P_1\in \sV_{n+k,k}\subset  \sV$.
Applying the function $F_{P_2}$ to both sides of \eqref{p0-p1}, we get
\begin{equation}\label{def-lambda}
  \lambda_0=-\frac{F_{P_2}(P_1)}{\ba_{P_2}\cdot \bv_{P_1}}.
\end{equation}
Hence for $1\le i\le d$, we have
\begin{equation}\label{long-frac}
\frac{p_{i,0}}{q_0}=\frac{p_{i,1}}{q_1}-\frac{F_{P_2}(P_1)}{\ba_{P_2}\cdot \bv_{P_1}}v_{i,1}=
\frac{\sum_{1\leq j\leq d}a_{j,2}(p_{i,1}v_{j,1}-p_{j,1}v_{i,1})-q_1v_{i,1}C_{P_2}}{q_1\ba_{P_2}\cdot \bv_{P_1}}.
\end{equation}
Now we claim that both the denominator and numerator of the fraction on the right hand side of \eqref{long-frac} are integers.

Indeed, by \eqref{v-b-z}, we have
\begin{equation*}
p_{i,1}v_{j,1}-p_{j,1}v_{i,1}= p_{i,1}z_j-p_{j,1}z_i\in \ZZ.\end{equation*}
Both the terms $q_1v_{i,1}$ and $q_1\ba_{P_2}\cdot \bv_{P_1}$ are easily seen to be integers by  \eqref{v-b-z}. This completes the proof of our claim.
Then it follows that
\begin{equation}\label{ine-funda}q_0\le q_1|\ba_{P_2}\cdot \bv_{P_1}|.\end{equation}
We have \begin{eqnarray}\label{ine-av}
&&\quad |\ba_{P_2}\cdot \bv_{P_1}|\notag\\
&&=|\sum_{1\leq i\leq d} a_{i,2}v_{i,1}|\notag\\
&&\leq^{} \sum_{1\le i\le d}\min\left\{q_2^{r_i},\xi_{P_2}\right\}\cdot(d-1)w_{i,1}q_1^{-r_i} \qquad (\text{by } \eqref{ine-def-v})\notag\\
&&=\sum_{1\le i\le d}\min\{q_1^{-r_i}q_2^{r_i},q_1^{-r_i}\xi_{P_2}\}\cdot(d-1)w_{i,1}\notag\\
&&\le \sum_{1\le i\le d}d\min\{R^{dr_i}, R^{dr_i}q_2^{-1-r_i}H(P_2)\}w_{i,1} \quad (\text{by definitions in Section \ref{SS:subdivision}})\notag\\
&&\le \sum_{1\le i\le d}dR^{d}\min\{1, q_2^{-1-r_i}H(P_2)\}w_{i,1}\notag\\
&&\le d^2R^d\max\{\psi_{P_1}, \psi_{P_2} \} \qquad (\text{by definitions of $\psi_P$})\notag\\
&&\le^{} d^2R^{d-dk-10d^2} \qquad (\text{by } \eqref{ine qxi})\notag\\
&&<R^{-dk-6d^2}.\end{eqnarray}
In particular \begin{equation}\label{q0-q1}\frac{q_0}{q_1}\leq
|\ba_{P_2}\cdot \bv_{P_1}|<R^{-dk-6d^2}<R^{-1}<\frac{1}{2}.\end{equation}

For any $1\le i\le d$, we have \begin{eqnarray}\label{ine-key2}
&&\quad q_{0}^{1+r_i}\Big|\frac{p_{i,0}}{q_0}-\frac{p_{i,1}}{q_1}\Big|\notag\\
&&=q_{0}^{1+r_i}\Big|\frac{F_{P_2}(P_1)}{\ba_{P_2}\cdot \bv_{P_1}}v_{i,1}\Big| \qquad (\text{by }\eqref{p0-p1} \text{ and } \eqref{def-lambda})\notag\\
&&\leq q_{1}|F_{P_{2}}(P_1)|\cdot|\ba_{P_2}\cdot \bv_{P_1}|^{r_i} \cdot(d-1)w_{i,1} \qquad (\text{by } \eqref{ine-def-v} \text{ and } \eqref{ine-funda}) \notag\\
&&\leq d q_{1}|F_{P_{2}}(P_1)|\cdot\max\left\{|\ba_{P_2}\cdot \bv_{P_1}|^{s}, \psi_{P_1}\right\} \qquad (\text{by definitions of $\psi_P$}) \notag\\
&&\leq (3d^{2}R^{k+d+1}c)\cdot(R^{-dk-6d^{2}})^{\frac{1}{d}} \qquad (\text{by }\eqref{ine-dist} \text{ and } \eqref{ine-av})\notag\\
&&\leq 3d^3R^{1-5d}c\notag\\
&&\le \frac{c}{2} \qquad (\text{by } \eqref{R-large}).
\end{eqnarray}

 Now we are ready to prove claim \eqref{containment}. Indeed, for any
$\bx=(x_{1},\dots,x_{d})\in\Delta_{c}(P_1)$ and any $1\le i\le d$,
\begin{eqnarray*}
&&\quad|q_0^{1+r_i}(x_i-\frac{p_{i,0}}{q_0})|\\
&&\le |q_0^{1+r_i}(x_{i}-\frac{p_{i,1}}{q_1})|+|q_0^{1+r_i}(\frac{p_{i,1}}{q_1}-\frac{p_{i,0}}{q_0})|\\
&&<(\frac{q_0}{q_1})^{1+r_i}c+\frac{c}{2} \qquad (\text{by } \ref{ine-key2})\\
&&<^{} c \qquad (\text{by } \ref{q0-q1}).
\end{eqnarray*}

\end{proof}

\begin{corollary}\label{P:main2}
Let $n\ge1$, $B\in\sB_n$ and $k\ge 1$. There exists a hyperplane $E_k(B)\subset\RR^{d}$ such that for any
$P\in\sV_{n+k,k}$, $$\Delta_c(P)\cap B\subset E_k(B)^{(R^{-(n+k)}\rho_0)}.$$
\end{corollary}

\begin{proof}
Choose $$\tP=\frac{\tpp}{\tq}=(\frac{\tp_1}{\tq},\dots,\frac{\tp_d}{\tq})\in\sV_{n+k,k}(B)$$ such that
$$\tq=q(\tP)=\min\{q(P): P\in\sV_{n+k,k}(B)\}.$$
Let $$\cH_{\tP}=\{\bx=(x_1,\dots,x_{d})\in\RR^{d}: F_{\tP}(\bx)= \sum_{1\leq i\leq d}\ta_ix_i+C_{\tP}=0\}$$
be the associated hyperplane. We show that $\cH_{\tP}$ is the hyperplane $E_k(B)$ we need, in other words
$$\Delta_c(P)\cap B\subset \cH_{\tP}^{(R^{-(n+k)}\rho_0)}.$$ For any $P\in\sV_{n+k,k}$,
if $P\notin\sV_{n+k,k}(B)$, this assertion is trivial. If
$$P=\frac{\bp}{q}=(\frac{p_1}{q},\dots,\frac{p_d}{q})\in\sV_{n+k,k}(B),$$ then $P\in \cH_{\tP}$ by Proposition
\ref{P:main prop}. For any $\bx=(x_1,\dots,x_{d})\in\Delta_c(P)\cap B$, its distance to the hyperplane
$\cH_{\tP}$ \begin{eqnarray*}
&&\quad \frac{1}{\sqrt{\sum_{1\leq i\leq d}\ta_i^{2}}}|F_{\tP}(\bx)|\\
&&\leq \frac{1}{\xi_{\tP}}\Big|\sum_{1\leq i\leq d} \ta_i(x_i-\frac{p_i}{q})\Big|\\
&&\leq\frac{1}{\xi_{\tP}}\sum_{1\leq i\leq d}\frac{|c\ta_i|}{q^{1+r_i}}\\
&&\leq \frac{1}{\xi_{\tP}}\sum_{1\leq i\leq d}\frac{c\tq^{r_i}}{q^{1+r_i}}\\
&&\leq \frac{dc}{\tq\xi_{\tP}}\\
&&\leq\frac{dc}{H_{n+k}}\\
&&\leq R^{-(n+k)}\rho_0.
\end{eqnarray*}
This finishes the proof.
\end{proof}

\section{Proof of Lemma \ref{l-b-hpw}}\label{S:main}
We are now ready to prove Lemma \ref{l-b-hpw}.
\begin{proof}[Proof of Lemma \ref{l-b-hpw}] It suffices to show that $\Bad_c(\rr)$ is $(B_0,\beta,\gamma)$-HPW. Denote the closed ball chosen by Bob at the $i$-th round as $B_i$ with radius $\rho_i$. By \cite[Remark 2.4]{AGK}, we may assume that $\rho_i\rightarrow 0$.
In view of Lemma \ref{L:decomposition}(1) and Lemma \ref{L:nouse}, we have
\begin{equation}\label{e:bad decomp}
 \Bad_c(\rr)= \RR^d\setminus \bigcup_{P\in\QQ^{d}}\Delta_{c}(P)=\RR^d\setminus \bigcup_{n=0}^{\infty}
\bigcup_{k=1}^{\infty}\bigcup_{P\in\sV_{n+k,k}}\Delta_{c}(P).
\end{equation}

Recall the definition of $\sB_n$ from \eqref{def-b-n}, as those $\sB_n$ are mutually disjoint, for each $i\ge 1$ there exists at most one $n\geq 1$ with $B_{i}\in \sB_n$.
According to the definition of $(\beta,\gamma)$-hyperplane potential game, we have $\rho_{i+1}\geq\beta \rho_{i}$.
Hence for each $n\geq 1$, there exists an $i\geq 1$ with $B_{i}\in\sB_n$. Let $i(n)$ denote the smallest $i$ with
$B_{i}\in\sB_n$. Then, the map $n\mapsto i(n)$ is an injective map from $\ZZ_{\geq 1}$ to $\mathbb{Z}_{\geq 1}$.
Let Alice play according to the following strategy: each time after Bob chooses a closed ball $B_i$, if $i=i(n)$
for some $n\geq 1$, then Alice chooses the family of hyperplane neighborhoods $$\{E_k(B_{i(n)})^{(R^{-(n+k)}\rho_0)}:
k\in\NN\}.$$ where $E_k(B_{i(n)})$ is the hyperplane given in Corollary \ref{P:main2}. Otherwise Alice makes
an arbitrary legal move. Since $B_{i(n)}\in\sB_n$, $\rho_{i(n)}>\beta R^{-n}\rho_0$. Then, \eqref{Rvalues} implies that
\begin{eqnarray*}&&\quad\sum_{k=1}^\infty(R^{-(n+k)}\rho_0)^\gamma\\&&=(R^{-n}\rho_0)^\gamma(R^\gamma-1)^{-1}
\\&&\leq (\frac{\rho_{i}}{\beta})^{\gamma}(\frac{\beta^2}{2})^{\gamma}\\&&<(\beta \rho_{i})^\gamma.\end{eqnarray*}
Hence \eqref{HPW ine} is satisfied, and Alice's move is legal. According to \eqref{e:bad decomp} and Corollary
\ref{P:main2}, we have \begin{eqnarray*}&&\quad \bigcap_{i=0}^\infty B_i\\
&&=\bigcap_{i=0}^{\infty}B_i\cap\big(\Bad(\rr)\cup(\bigcup_{n=1}^{\infty}
\bigcup_{k=1}^{\infty}\bigcup_{P\in\sV_{n+k,k}}\Delta_{c}(P)\big)\\
&&\subset\Bad(\rr)\cup\big(\bigcup_{n=1}^{\infty}\bigcup_{k=1}^{\infty}
\bigcup_{P\in\sV_{n+k,k}}\Delta_{c}(P)\cap B_{i(n)}\big)\\
&&=\Bad(\rr)\cup\big(\bigcup_{n=1}^{\infty}\bigcup_{k=1}^{\infty}
\bigcup_{P\in\sV_{n+k,k}(B_{i(n)})}\Delta_{c}(P)\cap B_{i(n)}\big)\\
&&\subset\Bad(\rr)\cup \big(\bigcup_{n=1}^{\infty}\bigcup_{k=1}^\infty E_k(B_{i(n)})^{(R^{-(n+k)}\rho_0)}\big).
\end{eqnarray*} Thus the unique point $\bx_\infty\in\bigcap_{i=0}^\infty B_i$ lies
in $$\Bad(\rr)\cup \big(\bigcup_{n=1}^{\infty}\bigcup_{k=1}^\infty
E_k(B_{i(n)})^{(R^{-(n+k)}\rho_0)}\big).$$
Hence, Alice wins.
\end{proof}






\begin{thebibliography}{99}

\bibitem{An1} J.~An, \textit{Badziahin-Pollington-Velani's theorem and Schmidt's game},
Bull. Lond. Math. Soc. (2013), no.\ 4, 721-733.

\bibitem{An2} J.~An, \textit{Two-dimensional badly approximable vectors and
Schmidt's game}, Duke Math. J. \textbf{165}  (2016), no. 2, 267-284..

\bibitem {An-B-V} J.~An, V.~Beresnevich, S.~Velani, \emph{Badly approximable points on
planar curves and winning}, Preprint, arXiv:1409.0064.

\bibitem {AGK} J.~An, L.~Guan, D.~Kleinbock, \textit{Bounded orbits of diagonalizable flows on
$ \SL_3(\RR)/\SL_3(\ZZ)$}, Internat. Math. Res. Notices.  (2015), no. 24, 13623-13652..


\bibitem{BPV} D.~Badziahin, A.~Pollington, S.~Velani, \textit{On a problem in simultaneous Diophantine
approximation: Schmidt's conjecture}, Ann. of Math. (2) \textbf{174} (2011), no. 3, 1837-1883.

\bibitem{Beresnevich} V.~Beresnevich, \emph{Badly approximable points on manifolds}, Invent. Math.
\textbf{202}  (2015), no. 3, 1199-1240.

\bibitem{BFKRW} R.~Broderick, L.~Fishman, D.~Kleinbock, A.~Reich, B.~Weiss,
\textit{The set of badly approximable vectors is strongly $C^1$ incompressible},
Math.\ Proc.\ Cambridge Philos.\ Soc.\ {\bf 153} (2012), no.\ 2, 319-339.

\bibitem{Da1} S.G.~Dani, \textit{Divergent trajectories of flows on homogeneous spaces and Diophantine
approximation}, J. Reine Angew. Math. \textbf{359} (1985), 55-89.




\bibitem{FSU} L.~Fishman, D.~Simmons, M.~Urba\'{n}ski, \textit{Diophantine
approximation and the geometry of limit sets in Gromov hyperbolic metric spaces},
Preprint, arXiv:1301.5630.


\bibitem{Kl} D.~Kleinbock, \textit{Flows on homogeneous spaces and Diophantine
properties of matrices}, Duke Math. J. \textbf{95} (1998), no. 1, 107-124.







\bibitem{KW3} D.~Kleinbock, B.~Weiss, \textit{Values of binary quadratic forms at
integer points and Schmidt games}, Recent Trends in Ergodic Theory and Dynamical Systems (Vadodara, 2012) (2013), 77-92.

\bibitem{Mc}  C.T.~McMullen, \textit{Winning sets, quasiconformal maps and Diophantine
approximation}, Geom.\ Funct.\ Anal.\ {\bf 20} (2010), no. 3, 726-740.


\bibitem {NS} E.~Nesharim, D.~Simmons, \textit{$\Bad(s,t)$ is hyperplane absolute
winning},  Acta Arith. 164 (2014), no. 2, 145-152.


\bibitem {Sc1} W.M.~Schmidt, \textit{On badly approximable numbers and certain games},
Trans. Amer. Math. Soc. \textbf{123} (1966), 178-199.

\bibitem{Sc1.5} W.M.~Schmidt, \textit{Badly approximable systems of linear forms},
J. Number Theory \textbf{1} (1969), 139-154.

\bibitem{Sc2} W.M.~Schmidt, \textit{Diophantine approximation}, Lecture Notes in
Mathematics \textbf{785}, Springer, Berlin, 1980.

\bibitem{Sc3} W.M.~Schmidt, \textit{Open problems in Diophantine approximation},
in ``Diophantine approximations and transcendental numbers (Luminy, 1982)", Progr.
Math. \textbf{31}, Birkh\"{a}user, Boston, 1983, pp. 271-287.

\end{thebibliography}
\end{document}